\documentclass[a4paper,10pt]{amsart}
\usepackage[top=0.8in, bottom=0.8in, left=1.25in, right=1.25in]{geometry}
\usepackage{amsmath}
\usepackage{amsfonts}
\usepackage{amssymb}
\usepackage{graphicx}
\usepackage{hyperref}
\usepackage{mathrsfs}
\usepackage{pb-diagram}
\usepackage{epstopdf}
\usepackage{CJK}
\numberwithin{equation}{section}
\allowdisplaybreaks[4]

\usepackage{amscd,amsthm,curves,enumerate,latexsym}
\usepackage{xypic}
\usepackage[all]{xy}
\usepackage{epstopdf}

\newtheorem{theorem}{Theorem}[section]
\newtheorem{lemma}[theorem]{Lemma}
\newtheorem{proposition}[theorem]{Proposition}
\newtheorem{exprop}[theorem]{Example-Proposition}

\newtheorem{corollary}[theorem]{Corollary}
\newtheorem{conjecture}[theorem]{Conjecture}
\newtheorem{question}[theorem]{Question}

\theoremstyle{definition}
\newtheorem{definition}[theorem]{Definition}

\theoremstyle{remark}
\newtheorem{remark}[theorem]{Remark}

\newcommand{\bP}{\mathbb{P}}

\newcommand\OO{{\mathcal{O}}}
\newcommand{\rounddown}[1]{\lfloor{#1}\rfloor}

\newcommand{\bC}{{\mathbb C}}

\newcommand\Pic{\text{\rm Pic}}

\newcommand\PP{{\mathcal{P}}}

\newcommand\ZZ{{\mathbb{Z}}}
\newcommand\bR{{\mathbb{R}}}

\newcommand\s{{ {\mathcal S}}}

 \newcommand\RR{{\mathcal{R}}}

 \newcommand{\MM}{{\mathcal M}}

\newcommand\Td{{\rm{Td}}}
\newcommand\td{{\rm{td}}}
\newcommand\Hilb{{\rm{Hilb}}}
\newcommand\Sym{{\rm{Sym}}}
\newcommand\ch{{\rm{ch}}}
\newcommand\bn{{\rm{\bf n}}}
\newcommand\bq{{\rm{\bf q}}}
\newcommand\pp{{\rm{p}}}

\begin{document}
\title{Remarks on Kawamata's effective non-vanishing conjecture for manifolds with trivial first Chern classes}
\author{Yalong Cao}
\address{Kavli Institute for the Physics and Mathematics of the Universe (WPI),The University of Tokyo Institutes for Advanced Study, The University of Tokyo, Kashiwa, Chiba 277-8583, Japan}
\email{yalong.cao@ipmu.jp}

\author{Chen Jiang}
\address{Shanghai Center for Mathematical Sciences, Fudan University, Jiangwan Campus, 2005 Songhu Road, Shanghai, 200438, China}
\email{chenjiang@fudan.edu.cn}

\thanks{Both authors were supported by World Premier International Research Center Initiative (WPI), MEXT, Japan. The second author was supported by JSPS KAKENHI Grant Number JP16K17558.}

\maketitle
\pagestyle{myheadings} \markboth{\hfill Y. Cao \& C. Jiang
\hfill}{\hfill Remarks on Kawamata's conjecture\hfill}
\begin{abstract}
Kawamata proposed a conjecture predicting that every nef and big line bundle on a smooth projective variety with trivial first Chern class has nontrivial global sections. We verify this conjecture for several cases, including (i) all hyperk\"{a}hler varieties of dimension $\leq 6$; (ii) all known hyperk\"{a}hler varieties except for O'Grady's 10-dimensional example; (iii) general complete intersection Calabi--Yau varieties in certain Fano manifolds (e.g. toric ones). Moreover, we investigate the effectivity of Todd classes of hyperk\"{a}hler varieties and Calabi--Yau varieties. We prove that the fourth Todd classes are ``quasi-effective" for all hyperk\"{a}hler varieties and general complete intersection Calabi--Yau varieties in products of projective spaces.
\end{abstract}


\section{Introduction}
Throughout this paper, we work over the complex number field $\bC$.

Kodaira vanishing theorem asserts that $H^{i}(X,K_{X}\otimes L)=0$ for $i>0$ and any ample line bundle $L$ on a smooth projective variety $X$. It is natural to ask when $H^0(X, K_{X}\otimes L)$ does not vanish. For example, Fujita's freeness conjecture predicts that $|K_X\otimes L^{\otimes (\dim X+1)}|$ is base-point-free (hence has many global sections).
Kawamata proposed the following conjecture in general settings. 
\begin{conjecture}[{Effective Non-vanishing Conjecture, \cite[Conjecture 2.1]{kawamata2}, see also \cite{Ambro99}}]\label{kawamata conj0}
Let $X$ be a projective normal variety with at most log terminal singularities, $L$ a Cartier divisor on $X$. Assume that $L$ is nef
and $L-K_{X}$ is nef and big. Then $H^{0}(X,L)\neq0$.
\end{conjecture}

Recall that a line bundle $L$ on $X$ is said to be {\it nef} if $L\cdot C\geq 0$ for any curve $C\subset X$, moreover, it is said to be {\it big} if $L^{\dim X}>0$.

We are interested in this conjecture for manifolds with trivial first Chern classes. In particular, Kawamata's Effective Non-vanishing Conjecture predicts the following conjecture:

\begin{conjecture}\label{kawamata conj}
Let $X$ be a smooth projective variety with $c_1(X)=0$ in $H^2(X, \bR)$ and $L$ a nef and big line bundle on $X$. Then $H^{0}(X,L)\neq0$.
\end{conjecture}

Thanks to Yau's proof of Calabi conjecture \cite{yau}, we have the so-called Beauville--Bogomolov decomposition (\cite{beauville, bogomolov}) for compact K\"{a}hler manifolds with trivial first Chern classes, which allows us to reduce Conjecture \ref{kawamata conj} to the cases of Calabi--Yau varieties and hyperk\"{a}hler varieties.
\begin{proposition}[{=Proposition \ref{reduce to irreducible}}]
Conjecture \ref{kawamata conj} is true if it holds for all Calabi--Yau varieties and hyperk\"{a}hler varieties.
\end{proposition}
In this paper, a {\it Calabi--Yau variety} is a compact K\"{a}hler manifold $X$ of dimension $n \geq 3$ with trivial canonical bundle such
that $h^0(\Omega^p
_X) = 0$ for $0 < p < n$ (since $h^{0,2}(X)=0$, $X$ is automatically a smooth projective variety by Kodaira embedding and Chow lemma), and a {\it hyperk\"ahler variety} is a simply connected smooth projective variety $Y$ such that $H^0(Y, \Omega^2_Y)$ is spanned by a non-degenerate two form.

For hyperk\"ahler varieties, the only known examples are (up to deformations): Hilbert schemes of points on K3 surfaces, generalized Kummer varieties (due to Beauville's construction \cite{beauville}), and two examples in dimension $6$ and $10$ introduced by O'Grady \cite{ogrady1, ogrady2}. As the first main theorem, we can verify Conjecture \ref{kawamata conj} for all hyperk\"{a}hler varieties of dimension up to six:

\begin{theorem}[= Corollary \ref{verification for six dim}]
Conjecture \ref{kawamata conj} is true for all hyperk\"{a}hler varieties of dimension at most $6$. 
\end{theorem}

As there are many known results on known hyperk\"ahler varieties, we can verify Conjecture \ref{kawamata conj} for those varieties (except for  O'Grady 10-fold) by easy computation, see Proposition \ref{verification for beauville} for more details.
\begin{exprop}[= Proposition \ref{verification for beauville}]
Conjecture \ref{kawamata conj} is true for the following cases:
\begin{enumerate}
\item hyperk\"{a}hler varieties homeomorphic to Hilbert schemes of points on K3 surfaces;
\item hyperk\"{a}hler varieties homeomorphic to generalized Kummer varieties;
\item hyperk\"{a}hler varieties homeomorphic to O'Grady's $6$-dimensional example. 
\end{enumerate}
\end{exprop}

For Calabi--Yau varieties, we have lots of examples provided by smooth complete intersections in Fano manifolds.
To state our result, we introduce a subclass of Fano manifolds, which we call \emph{perfect} Fano manifolds, i.e. on which any nef line bundle has a nontrivial section (Definition \ref{def of perfect Fano}). This subclass contains, for example, toric Fano manifolds, Fano manifolds of dimension $n\leq4$, and their products (see Proposition \ref{prop on perfect Fano}). On the other hand, Kawamata's Effective Non-vanishing Conjecture predicts that every Fano manifold is perfect. Then we show that Conjecture \ref{kawamata conj} is true for general complete intersection Calabi--Yau varieties in perfect Fano manifolds.
\begin{theorem}[{=Theorem \ref{verification for CICY}}]
Let $Y$ be a perfect Fano manifold, $\{H_{i}\}_{i=1}^{m}$ a sequence of base-point-free ample line bundles on $Y$ such that $\otimes_{i=1}^{m}H_{i}\cong K_{Y}^{-1}$. Let $X\subseteq Y$ be a general complete intersection of dimension $n\geq3$ defined by common zero locus of sections of $\{H_{i}\}_{i=1}^{m}$.
Then Conjecture \ref{kawamata conj} is true for $X$.
\end{theorem}
Here ``general" means that if we label the sections by $\{s_{i}\}_{i=1}^{m}$, we require $\{s_{1}=s_{2}=\cdots=s_{i}=0\}$ are smooth for $i=1,2,...,m$.

By Hirzebruch--Riemann--Roch formula, the existence of global sections is related to the effectivity of Todd classes. We define the following quasi-effectivity for cycles.
\begin{definition}Let $X$ be a smooth projective variety and $\gamma$ an algebraic $k$-cycle. $\gamma$ is said to be {\it quasi-effective} if the intersection number $\gamma\cdot L^{k}\geq 0$ for any nef line bundle $L$.
\end{definition}
We remark that to test quasi-effectivity, it suffices to check for only ample (or nef and big) line bundles $L$ since nef line bundles can be viewed as ``limits" of ample line bundles \cite{lazarsfeld}.

As explained above, we will consider the following question.
\begin{question}\label{question}
Are Todd classes of a Calabi--Yau variety or a hyperk\"ahler variety quasi-effective?\end{question}
It is well-known that the second Todd class of a smooth projective variety with $c_1=0$ is quasi-effective (which is equivalent to pseudo-effectivity in this case) by the result of Miyaoka and Yau \cite{miyaoka, yau0}.
We consider higher Todd classes and answer this question affirmatively in the following cases.
\vspace{7pt}
\begin{theorem}[{=Theorem \ref{effective of Td4}, Proposition \ref{verification for beauville}, and Theorem \ref{CICY td4}}]~
\begin{enumerate}
\item The fourth Todd classes of all hyperk\"ahler varieties are quasi-effective;
\item All Todd classes of hyperk\"ahler varieties homeomorphic to any known hyperk\"ahler variety, except for O'Grady's 10-dimensional example remaining unknown, are quasi-effective;
\item The fourth Todd classes of general complete intersection Calabi--Yau in products of projective spaces are quasi-effective.
\end{enumerate}
\end{theorem}
We also prove a weaker version of Conjecture \ref{kawamata conj} which relates to a conjecture of Beltrametti and Sommese (see H\"{o}ring's work \cite{horing}).
\begin{theorem}[{=Theorems \ref{HRR for odd CY}, \ref{HRR for even CY}}]
Fix an integer $n\geq 2$. Let $X$ be an $n$-dimensional smooth projective variety with $c_1(X)=0$ in $H^2(X, \bR)$ and $L$ a nef and big line bundle on $X$.
Then there exists a positive integer $i\leq \rounddown{\frac{n-1}{4}}+\rounddown{\frac{n+2}{4}}$ such that $H^{0}(X,L^{\otimes i})\ne 0$.

In particular, Conjecture \ref{kawamata conj} is true in dimension $\leq4$.
\end{theorem}
Finally, we remark that one could also consider the K\"{a}hler version of Conjecture \ref{kawamata conj} and Question \ref{question}. However, they can be simply reduced to the projective case. As for Conjecture \ref{kawamata conj}, the existence of a nef and big line bundle on a compact complex manifold forces the manifold to be Moishezon (e.g. \cite[Theorem 2.2.15]{ma}), and hence projective since it is also K\"{a}hler. For Question \ref{question}, if we consider a non-projective hyperk\"ahler manifold $X$, then $q_X(L) =0$ for every nef line bundle $L$ (\cite[Theorem 3.11]{huybrechts}) and hence the intersection numbers of nef line bundles with Todd classes are identically zero by Fujiki's result \cite{fujiki} (see Theorem \ref{fujiki result}).

\vspace{1em}

\textbf{Acknowledgement.} The second author is grateful to Professor Keiji Oguiso for discussions. Part of this paper was written when the second author was visiting National Center for Theoretical Sciences and he would like to thank Professors Jungkai Chen and Ching-Jui Lai for their hospitality. The authors are grateful to the referees for many useful suggestions which improve the presentation of this paper.

\section{Reduction to Calabi--Yau and hyperk\"{a}hler varieties}
Recall the following Beauville--Bogomolov decomposition thanks to Yau's proof of
Calabi conjecture \cite{yau}. 
\begin{theorem}[{Beauville--Bogomolov decomposition, \cite{beauville, bogomolov}}]\label{classification of CY}Every smooth projective variety $X$ with $c_1(X)=0$ in $H^2(X, \bR)$ admits an \'etale finite cover isomorphic to a product of Abelian varieties, Calabi--Yau varieties, and hyperk\"{a}hler varieties.
\end{theorem}
Then it is easy to reduce Conjecture \ref{kawamata conj} to the cases of Calabi--Yau varieties and hyperk\"{a}hler varieties.
\begin{proposition}\label{reduce to irreducible}
Conjecture \ref{kawamata conj} is true if it holds for all Calabi--Yau varieties and hyperk\"{a}hler varieties.
\end{proposition}
\begin{proof}
Let $X$ be a smooth projective variety with $c_1(X)=0$ in $H^2(X, \bR)$ and $L$ a nef and big line bundle on $X$. By Beauville--Bogomolov decomposition, there exists an \'etale finite cover $\pi:X^\prime\rightarrow X$ such that
$$
X^\prime \cong A\times X_1\times \cdots \times X_n\times Y_1\times \cdots \times Y_m,
$$
where $A$ is an Abelian variety, $X_i$ a Calabi--Yau variety, and $Y_j$ a hyperk\"{a}hler variety. After pull-back $L$ to $X^\prime$,
$$\chi(X^\prime,\pi^{*}L)=\deg(\pi)\cdot \chi(X,L). $$
and $\pi^{*}L$ is again nef and big on $X^\prime$. Kawamata--Viehweg
vanishing theorem (\cite{kawamata1}) implies
\begin{align*}
h^{0}(X^\prime,\pi^{*}L)={}&\chi(X^\prime,\pi^{*}L),\\
h^{0}(X,L)={}&\chi(X,L).
\end{align*}
Hence to prove $h^{0}(X,L)\neq 0$, it is equivalent to prove $h^{0}(X^\prime,\pi^{*}L)\neq 0$.
On the other hand, since by definition $h^1(X_i, \OO_{X_i})=h^1(Y_j, \OO_{Y_j})=0$ for all $i,j$, by \cite[Ex III.12.6]{hart}, there is a natural isomorphism
$$\Pic(X^\prime)\cong \Pic(A)\times\prod^{n}_{i=1}\Pic(X_{i})\times\prod^{m}_{j=1}\Pic(Y_{j}),$$ i.e.
a line bundle on $X^\prime$ is the box tensor of its restriction on each factors. Since the restriction of $\pi^{*}L$ on each factor is obviously nef and big, to prove $h^{0}(X^\prime,\pi^{*}L)\neq 0$, it suffices to show for any nef and big line bundle on each factor, there exists
a nontrivial global section, that is, to verify Conjecture \ref{kawamata conj} for Abelian varieties, Calabi--Yau varieties, and hyperk\"{a}hler varieties.
Note that Conjecture \ref{kawamata conj} is true for Abelian varieties by Kodaira vanishing and Hirzebruch--Riemann--Roch formula.
\end{proof}

\section{Hyperk\"{a}hler case}
\subsection{Preliminaries}Let $X$ be a hyperk\"ahler variety. Beauville \cite{beauville} and Fujiki \cite{fujiki} proved that there exists a quadratic form $q_X: H^2(X, \mathbb{C}) \to \mathbb{C}$ and a constant $c_X\in \mathbb{Q}_+$ such that for all $\alpha\in H^2(X, \mathbb{C})$,
$$\int_X\alpha^{2n}= c_X\cdot q_X(\alpha)^n. $$
The above equation determines $c_X$ and $q_X$ uniquely if assuming:
\begin{enumerate} 
\item $q_X$ is a primitive integral quadratic form on $H^2(X, \ZZ)$;
\item $q_X(\sigma + \overline{\sigma}) > 0$ for $0\neq \sigma \in H^{2,0}(X)$.
\end{enumerate}
Here $q_X$ and $c_X$ are called the {\it Beauville--Bogomolov--Fujiki form} and the {\it Fujiki constant} of $X$ respectively.
Note that condition (2) above is equivalent to the following condition (see \cite[Propositions 8 and 11]{nieper0}):
\begin{enumerate}
\item[(2')] There exists an  $\alpha \in H^2(X,\bR)$ with $q_X(\alpha)\neq 0$, and for all $\alpha \in H^2(X,\bR)$
with $q_X(\alpha)\neq 0$, we have that
 $$   
\frac{\int_X \pp_1(X)\alpha^{2n-2} }{q_X(\alpha)^{n-1}}< 0.$$
\end{enumerate}
Here $\pp_1(X)$ is the first Pontrjagin class. By the above definition, it is easy to see that $q_X$ and $c_X$ are actually topological invariants.

Recall a result by Fujiki \cite{fujiki} (see also \cite[Corollary 23.17]{gross} for a generalization).
\begin{theorem}[{\cite{fujiki}, \cite[Corollary 23.17]{gross}}]\label{fujiki result}
Let $X$ be a hyperk\"ahler variety of dimension $2n$. Assume $\alpha\in H^{4j}(X,\mathbb{C})$ is of type $(2j, 2j)$ on all small deformation of $X$. Then there exists a constant $C(\alpha)\in\mathbb{C}$ depending on $\alpha$ such that
$$\int_{X}\alpha\cdot\beta^{2n-2j}=C({\alpha})\cdot q_{X}(\beta)^{n-j},$$
for all $\beta\in H^{2}(X,\mathbb{C})$.
\end{theorem}
 A direct application of this result (cf. \cite[1.11]{huybrechts}) is that, for a line bundle $L$ on $X$,
Hirzebruch--Riemann--Roch formula gives
$$\chi(X,L)=\sum_{i=0}^{2n}\frac{1}{(2i)!}\int_X\td_{2n-2i}(X)(c_{1}(L))^{2i}=\sum_{i=0}^{2n}\frac{a_{i}}{(2i)!}q_{X}\big(c_{1}(L)\big)^{i}, $$ 
where $$a_{i}=C(\td_{2n-2i}(X)).$$
Since rational Chern classes are determined by rational Pontrjagin classes (cf. \cite[Proposition 1.13]{nieper1}), rational Chern classes (and hence Todd classes) are topological invariants of $X$ by Novikov's theorem, hence $a_i$'s in the above formula  are constants depending only on the topology of $X$.

For a line bundle $L$ on $X$, Nieper \cite{nieper} defined the {\it characteristic value} of $L$,
$$
\lambda(L):=\begin{cases}\frac{24n\int_{X}\ch(L)}{\int_{X}c_{2}(X) \ch(L)} & \text{if well-defined;}\\ 0 & \text{otherwise.}\end{cases}
$$
Note that $\lambda(L)$ is a positive (topological constant) multiple of $q_X(c_1(L))$ (cf. \cite[Proposition 10]{nieper}), more precisely,
$$
\lambda(L)=\frac{12c_X}{(2n-1)C(c_2(X))}q_X(c_1(L)).
$$

It is easy to see that if $L$ is a nef line bundle, then $q_X(c_1(L))\geq 0$ and $\lambda(L)\geq 0$; if $L$ is a nef and big line bundle, then $q_X(c_1(L))>0$ and $\lambda(L)>0$ (cf. \cite[Corollary 6.4]{huybrechts}).

\subsection{Quasi-effectivity of $4$-th Todd classes of hyperk\"{a}hler varieties}
In this subsection, we prove the following thereom.
\begin{theorem}\label{effective of Td4}
Let $X$ be a hyperk\"{a}hler variety of dimension $2n$ with $n\geq 2$. Then $\td_4(X)$ is quasi-effective, that is,
$$\int_{X}\td_{4}(X)\cdot L^{2n-4}\geq 0$$
for any nef line bundle $L$ on $X$. Moreover, the inequality is strict for nef and big line bundle $L$.
\end{theorem}
Firstly, we prove the following lemma.
\begin{lemma}[{cf. \cite[Lemma 1]{kurnosov}, \cite[Lemma 3]{guan}}]\label{lemma c22}
Let $X$ be a hyperk\"{a}hler variety of dimension $2n$ with $n\geq 2$. Then $c_2^2(X)$ is quasi-effective, that is,
$$\int_{X}c_2^2(X)\cdot L^{2n-4}\geq 0$$
for any nef line bundle $L$ on $X$. Moreover, the inequality is strict for nef and big line bundle $L$.
\end{lemma}
\begin{proof}By Theorem \ref{fujiki result},
$$
\int_{X}c_2^2(X)\cdot L^{2n-4}=C(c_2^2(X))\cdot q_X(c_1(L))^{n-2}
$$
for a line bundle $L$.
Hence it is equivalent to prove that $C(c_2^2(X))>0$.

Fix $0\ne \sigma\in H^{2,0}(X)$, by Theorem \ref{fujiki result},
$$
\binom{2n-4}{n-2}\int_{X}c_2^2(X)\cdot (\sigma\overline{\sigma})^{n-2}=C(c_2^2(X))\cdot q_X(\sigma+\overline{\sigma})^{n-2}.
$$
Hence it is equivalent to prove that $\int_{X}c_2^2(X)\cdot (\sigma\overline{\sigma})^{n-2}>0$.

Take $Q\in \Sym^2H^2(X)$ to be the dual of Beauville--Bogomolov--Fujiki form $q_X$. Taking the orthogonal decomposition of $c_2\in H^4(X)$ induced by the projection to $\Sym^2H^2(X)$, we may write $c_2=\mu Q+p$, where $\mu>0$ is a positive rational number (cf. \cite[Proposition 12]{nieper0}) and $p\in H^4_{\text{prim}}(X)$. Since $p$ is a $(2,2)$-form, by the Hodge--Riemann bilinear relations,
\begin{align*}
\int_{X}c_2^2(X)\cdot (\sigma\overline{\sigma})^{n-2}={}&\mu^2\int_{X}Q^2\cdot (\sigma\overline{\sigma})^{n-2}+2\mu\int_{X}Qp\cdot (\sigma\overline{\sigma})^{n-2}+\int_{X}p^2\cdot (\sigma\overline{\sigma})^{n-2}\\
={}&\mu^2\int_{X}Q^2\cdot (\sigma\overline{\sigma})^{n-2}+\int_{X}p^2\cdot (\sigma\overline{\sigma})^{n-2}\\
\geq {}&\mu^2\int_{X}Q^2\cdot (\sigma\overline{\sigma})^{n-2}.
\end{align*}
Hence it suffices to show that $\int_{X}Q^2\cdot (\sigma\overline{\sigma})^{n-2}>0$. Since $q_X$ is a deformation invariant, so is $Q^2$. By Theorem \ref{fujiki result}, it suffices to show that $C(Q^2)>0$.

Let $e_1,\ldots, e_{b_2}$ be an orthonormal basis on $H^2(X,\bC)$ for which $Q = \sum_{i=1}^{b_2} e_i^2.$ Then for distinct integers $i,j,k$, and formal parameters $t,s$,
$$
\int_X (e_i+te_j+se_k)^{2n}=c_X\cdot q_X(e_i+te_j+se_k)^n=c_X\cdot (1+t^2+s^2)^n.
$$
Comparing the coefficients of $t, s$,
\begin{align*}
{}&\int_X e_i^{2n}=c_X,\\
{}&\int_X e_i^{2n-2}e_j^2=\frac{c_X}{2n-1},\\
{}&\int_X e_i^{2n-4}e_j^4=\frac{3c_X}{(2n-1)(2n-3)},\\
{}&\int_X e_i^{2n-4}e_j^2e_k^2=\frac{c_X}{(2n-1)(2n-3)}.
 \end{align*}
 Hence by Theorem \ref{fujiki result} and $q_X(e_1)=1$,
\begin{align*}
C(Q^2)=\int_XQ^2e_1^{2n-4}=c_X+\frac{2(b_2-1)c_X}{2n-1}+\frac{3(b_2-1)c_X}{(2n-1)(2n-3)}+\frac{(b_2-1)(b_2-2)c_X}{(2n-1)(2n-3)}>0.
\end{align*}
We complete the proof.
 \end{proof}
\begin{proof}[Proof of Theorem \ref{effective of Td4}]
By Nieper's formula (see \cite[Remark after Definition 19]{nieper} or \cite[Corollary 3.7]{nieper1}), for any $\alpha\in H^{2}(X)$,
\begin{equation}\int_{X}\sqrt{\td(X)}\exp(\alpha)=(1+\lambda(\alpha))^{n} \int_{X}\sqrt{\td(X)}.
\label{nieper formula} \end{equation}
Here$\sqrt{\td(X)}$ is defined to be the formal power series in cohomology ring whose square is $\td(X)$. Note that
$$
\sqrt{\td(X)}=1+\frac{1}{24}c_2(X)+\frac{1}{5760}(7c_2^2(X)-4c_4(X))+\cdots
$$
and $\int_{X}\sqrt{\td(X)}>0$ by a theorem of Hitchin and Sawon \cite{hitchinsawon}.

Fix a nef and big line bundle $L$ on $X$ (hence $\lambda(L)>0$), take $\alpha=t\cdot c_1(L)$ where $t$ is a formal parameter and compare the coefficients of $t$ in (\ref{nieper formula}), then
\begin{equation}\int_{X}\sqrt{\td(X)}\cdot L^{2n-4}=(2n-4)!\binom{n}{n-2}\lambda(L)^{n-2}\int_{X}\sqrt{\td(X)}>0.
\nonumber \end{equation}
Equivalently, this gives
\begin{equation}\int_{X}(7c_{2}^{2}(X)-4c_{4}(X))\cdot L^{2n-4}>0.
\nonumber \end{equation}
Combining with Lemma \ref{lemma c22},
\begin{equation}\int_{X}\td_{4}(X)\cdot L^{2n-4}=\frac{1}{2880}\int_{X}(7c_{2}^{2}(X)-4c_{4}(X))\cdot L^{2n-4}+\frac{1}{576}\int_{X}c_{2}^{2}(X)\cdot L^{2n-4}>0.  \nonumber \end{equation}
We complete the proof.
\end{proof}
\begin{corollary}\label{verification for six dim}
Let $X$ be a $6$-dimensional hyperk\"{a}hler variety and $L$ a nef and big line bundle on $X$. Then
\begin{equation}h^{0}(X,L)\geq 5. \nonumber \end{equation}
\end{corollary}
\begin{proof}
By Kawamata--Viehweg vanishing theorem and Hirzubruch--Riemann--Roch formula,
\begin{equation}h^{0}(X,L)=\chi(X,L)=\frac{1}{6!}\int_{X}c^{6}_{1}(L)+\frac{1}{4!}\int_{X}c^{4}_{1}(L)\cdot \td_{2}(X)+\frac{1}{2}\int_{X}c^{2}_{1}(L)\cdot \td_{4}(X)+\chi(X,\mathcal{O}_{X}). \nonumber \end{equation}
As it is shown by Miyaoka and Yau \cite{miyaoka, yau0} that $\td_2(X)=\frac{1}{12} c_2(X)$ is quasi-effective, along with Theorem \ref{effective of Td4} we get
$h^{0}(X,L)>\chi(X,\mathcal{O}_{X})=4$.
\end{proof}
\subsection{Known hyperk\"{a}hler varieties}
For hyperk\"{a}hler varieties, the only known examples are (up to deformations) two families, Hilbert schemes of points on K3 surfaces and generalized Kummer varieties, due to Beauville's construction \cite{beauville} and two examples introduced by O'Grady \cite{ogrady1, ogrady2}. We verify Conjecture \ref{kawamata conj} for these known examples except for O'Grady's $10$-dimensional example.
\begin{proposition}\label{verification for beauville}
Let $L$ be a nef and big line bundle on a hyperk\"{a}hler variety $X$ of dimension $2n$.
\begin{enumerate}
\item If $X$ is homeomorphic to a Hilbert scheme of points on K3 surfaces, then $$h^{0}(X,L)\geq \frac{(n+2)(n+1)}{2}.$$
\item If $X$ is homeomorphic to a generalized Kummer variety with $n\geq 2$, or O'Grady's $6$-dimensional example, then $$h^{0}(X,L)\geq (n+1)^{2}.$$
\end{enumerate}
Moreover, in the above cases, all Todd classes of $X$ are quasi-effective.
\end{proposition}
\begin{proof}
Let $X$ be a hyperk\"{a}hler manifold of dimension $2n$ and $L$ a line bundle.
By Theorem \ref{fujiki result}, we know that
$$
\chi(X, L)=P_X(q_X(c_1(L)))=\tilde{P}_X(\lambda(L)),
$$
where $P_X$ and $\tilde{P}_X$ are polynomials depending only on the topology of $X$. Note that all Todd classes of $X$ are quasi-effective if and only if all coefficients of $P_X$ (or $\tilde{P}_X$) are non-negative.

For a generalized Kummer variety $K^{n}A$ of an Abelian surface $A$, and a line bundle $H$ on $K^{n}A$, Britze--Nieper \cite{britze} showed that
\begin{equation}\chi(K^{n}A, H)=(n+1)\binom{\frac{(n+1)\lambda(H)}{4}+n}{n}.
\nonumber \end{equation}
Thus
$$
\tilde{P}_{K^{n}A}(t)=(n+1)\binom{\frac{(n+1)t}{4}+n}{n}$$ is a polynomial with positive coefficients.
Hence if $X$ is homeomorphic to $K^{n}A$ and $L$ is a line bundle on $X$,
then all Todd classes of $X$ are quasi-effective and
$$
\chi(X, L)=\tilde{P}_X(\lambda(L))=\tilde{P}_{K^{n}A}(\lambda(L))=(n+1)\binom{\frac{(n+1)\lambda(L)}{4}+n}{n}.
$$
When $L$ is nef and big, $\lambda(L)\in\mathbb{Q}_{>0}$ and hence by Lemma \ref{lemma on integers} (assuming that $n\geq 2$), $\frac{(n+1)\lambda(L)}{4}$ is a positive integer and $h^{0}(X, L )=\chi(X, L )\geq(n+1)^{2}$.

For a Hilbert scheme of $n$-points $\Hilb^{n}(S)$ on a K3 surface $S$,
\begin{equation}\Pic(\Hilb^{n}(S))\cong \Pic(S)\oplus \mathbb{Z}E \nonumber \end{equation}
and any line bundle on $\Hilb^{n}(S)$ is of the form $H_{n}\otimes E^{r}$, where $r\in\mathbb{Z}$ and $H_{n}$ is induced by a line bundle $H$ on $S$ (see \cite[Section 5]{egl}). Ellingsrud--G\"{o}ttsche--Lehn's formula (\cite[Theorem 5.3]{egl}) gives
\begin{equation}\chi(\Hilb^{n}(S),H_{n}\otimes E^{r})=\binom{\chi(S,H)-(r^{2}-1)(n-1) }{ n}.
\nonumber \end{equation}
With the Beauville--Bogomolov--Fujiki form $q$, we have
\begin{equation}\big(H^{2}(\Hilb^{n}(S),\mathbb{Z}),q\big)\cong H^{2}(S,\mathbb{Z})_{(-, -)}\oplus_{\perp}\mathbb{Z}[E], \nonumber \end{equation}
where $(-,-)$ is the cup product on $S$ and $q([E])=-2(n-1)$ (see \cite[Section 2.2]{huybrechts}). Thus
\begin{equation}q(c_{1}(H_{n}\otimes E^{r}))=(H)^2-2r^{2}(n-1), \nonumber \end{equation}
and
\begin{equation}\chi(\Hilb^{n}(S),H_{n}\otimes E^{r})=\binom{\frac{1}{2}q(c_{1}(H_{n}\otimes E^{r}))+n+1}{n}.
\nonumber \end{equation}
Thus
$$
P_{\Hilb^{n}(S)}(t)=\binom{\frac{1}{2}t+n+1}{n}$$
 is a polynomial with positive coefficients.
Hence if $X$ is homeomorphic to $\Hilb^{n}(S)$ and $L$ is a line bundle on $X$,
then all Todd classes of $X$ are quasi-effective and
$$
\chi(X, L)=P_X(q_X(c_1(L)))=P_{\Hilb^{n}(S)}(q_X(c_1(L)))=\binom{\frac{1}{2}q_X(c_1(L))+n+1}{n}.
$$
When $L$ is nef and big, $q_X(c_1(L))\in\mathbb{Q}_{>0}$ and hence by Lemma \ref{lemma on integers}, $\frac{1}{2}q_X(c_1(L))$ is a positive integer and $h^{0}(X, L )=\chi(X, L )\geq\binom{n+2}{n}$.


In general, for a hyperk\"{a}hler variety $X$ of dimension $2n$ and a line bundle $L$ on $X$, Nieper \cite{nieper} used Rozansky--Witten invariants to express those coefficients $a_{i}$ in the expansion of $\chi(X,L)$ in terms of Chern numbers of $X$. More precisely, we have
\begin{equation}\chi(X,L)=\int_{X}\exp\bigg(-2\sum_{k=1}^{\infty}\frac{B_{2k}}{4k}\ch_{2k}(X)T_{2k}\bigg(\sqrt{\frac{\lambda(L)}{4}+1}\bigg)\bigg),
\label{nieper RR} \end{equation}
where $B_{2k}$ are the Bernoulli numbers with $B_{2}=\frac{1}{6}$, $B_{4}=-\frac{1}{30}$, $B_{6}=\frac{1}{42}$, and $T_{2k}$ are even Chebyshev polynomials.

Now consider O'Grady's six dimensional hyperk\"{a}hler manifold $\MM_6$. By Mongardi--Rapagnetta--Sacc\`{a}'s computation
(\cite[Corollary 6.8]{mrs}), we have
\begin{equation}\int_{\MM_6}c_{2}^{3}(\MM_6)=30720, \quad \int_{\MM_6}c_{2}(\MM_6)c_{4}(\MM_6)=7680, \quad \int_{\MM_6}c_{6}(\MM_6)=1920.
\nonumber \end{equation}
After direct computations by (\ref{nieper RR}), for a line bundle $H$ on $\MM_6$,
\begin{equation}\chi(\MM_6,H)=\frac{4}{27}\Big(T_{6}(\sqrt{\lambda(H)/4+1})+6T_{2}(\sqrt{\lambda(H)/4+1})\cdot T_{4}(\sqrt{\lambda(H)/4+1})+20T^{3}_{2}(\sqrt{\lambda(H)/4+1})\Big). \nonumber \end{equation}
Plugging Chebyshev polynomials: $T_{2}(x)=2x^{2}-1$, $T_{4}(x)=8x^{4}-8x^{2}+1$, $T_{6}(x)=32x^{6}-48x^{4}+18x^{2}-1$ into the formula, we get
\begin{equation}\chi(\MM_6,H)=4\binom{\lambda(H)+3}{3}. \nonumber \end{equation}
Thus
$$
\tilde{P}_{\MM_6}(t)=4\binom{t+3}{3}$$
is a polynomial with positive coefficients.
Hence if $X$ is homeomorphic to $\MM_6$ and $L$ is a line bundle on $X$,
then all Todd classes of $X$ are quasi-effective and
$$
\chi(X, L)=\tilde{P}_X(\lambda(L))=\tilde{P}_{\MM_6}(\lambda(L))=4\binom{\lambda(L)+3}{3}.
$$
When $L$ is nef and big, $\lambda(L)\in\mathbb{Q}_{>0}$ and hence by Lemma \ref{lemma on integers}, $\lambda(L)$ is a positive integer and $h^{0}(X, L )=\chi(X, L )\geq 4\binom{4}{3}=16$.
\end{proof}




From the above computations, we may propose the following conjecture.
\begin{conjecture}
Let $X$ be a hyperk\"{a}hler variety of dimension $2n$ and $L$ a nef and big line bundle on $X$. Then
\begin{enumerate}
\item $h^{0}(X,L)\geq n+2$;
\item more wildly, $\int_{X}\td_{2n-2i}(X)\cdot L^{2i}\geq 0$ for all $i\in\mathbb{Z}$.
\end{enumerate}
\end{conjecture}

\section{Calabi--Yau case}
\subsection{Complete intersections in Fano manifolds}
As for Calabi--Yau varieties, there is a huge amount of examples provided by complete intersections in Fano manifolds.
In this subsection, we verify Conjecture \ref{kawamata conj} for these examples.

Let $X$ be a projective variety, $N_{1}(X)$ be its set of numerical equivalent classes of 1-cycles in $\mathbb{R}$-coefficients. Set
\begin{equation}NE(X)=\big\{\sum a_{i}[C_{i}]\in N_{1}(X) \mid C_{i}\subseteq X, \textrm{ } 0\leq a_{i}\in \bR \big\}.  \nonumber \end{equation}
and $\overline{NE}(X)$ its closure in $N_{1}(X)$. Note that $\overline{NE}(X)$ is the dual of the cone of nef divisors on $X$.

As our testing examples are realized as hypersurfaces, or more generally, complete intersections in Fano manifolds, we recall the following comparison result for closed cone of curves.
\begin{theorem}[{\cite{kollar}, \cite[Proposition 3.5]{beltrametti}}]\label{kollar thm}
Let $Y$ be a projective manifold of dimension $n\geq4$, $H$ be an ample line bundle on $Y$, and $X$ be an effective smooth divisor in $|H|$.
Assume $K^{-1}_{Y}\otimes H^{-1}$ is nef. Then $\overline{NE}(X)\cong\overline{NE}(Y)$ under the natural embedding $X\hookrightarrow Y$.
\end{theorem}
In fact, by the Lefschetz hyperplane theorem (see e.g. \cite[Example 3.1.25]{lazarsfeld}), $\Pic(X)\cong \Pic(Y)$ under the embedding $i: X\hookrightarrow Y$, and $i_{*}: \overline{NE}(X)\rightarrow \overline{NE}(Y)$ is an inclusion. The above theorem says that under certain condition, $i_{*}$ is an isomorphism and nef line bundles on $X$ and $Y$ are identified under $i^*$. We then compare sections of those line bundles on $X$ and $Y$.
\begin{proposition}\label{compare sections}
Let $Y$ be a projective manifold, $H$ a line bundle such that $h^{0}(Y,H)\geq2$. Suppose that the linear system $|H|$ contains a smooth element $X$.
Assume that $K^{-1}_{Y}\otimes H^{-1}$ is nef. Then for any nef line bundle $L$ on $Y$ with $L\otimes (K^{-1}_{Y}\otimes H^{-1})$ big, we have
\begin{equation}h^{0}(Y,L)=h^{0}(X,L|_{X})+h^{0}(Y,L\otimes H^{-1}). \nonumber \end{equation}
Furthermore, if $h^{0}(Y,L)>0$, then $h^{0}(X,L|_{X})>0$.
\end{proposition}
\begin{proof}
Twisting the exact sequence
\begin{equation}0\rightarrow\mathcal{O}_{Y}(-X)\rightarrow\mathcal{O}_{Y}\rightarrow\mathcal{O}_{X}\rightarrow0 \nonumber \end{equation}
with $L$, and taking long exact sequence of the corresponding cohomology, we obtain
\begin{equation}0\rightarrow H^{0}(Y,L\otimes H^{-1})\rightarrow H^{0}(Y,L)\rightarrow H^{0}(X,L|_{X})\rightarrow H^{1}(Y,L\otimes H^{-1}). \nonumber \end{equation}
Since $L\otimes (K^{-1}_{Y}\otimes H^{-1})$ is nef and big, $H^{1}(Y,L\otimes H^{-1})=0$ by Kawamata--Viehweg vanishing theorem \cite{kawamata1}.
Thus
\begin{equation}h^{0}(X,L|_{X})=h^{0}(Y,L)-h^{0}(Y,L\otimes H^{-1}). \nonumber \end{equation}

Assume that $h^{0}(Y,L)>0$.
If $h^{0}(Y,L\otimes H^{-1})=0$, then $$h^{0}(X,L|_{X})=h^{0}(Y,L)>0. $$
If $h^{0}(Y,L\otimes H^{-1})> 0$, then by \cite[Lemma 15.6.2]{kol shafa}, $$h^{0}(Y,L)\geq h^{0}(Y,H)+h^{0}(Y,L\otimes H^{-1})-1\geq h^{0}(Y,L\otimes H^{-1})+1,$$
and hence $h^{0}(X,L|_{X})\geq1$.
\end{proof}
Thus, to prove Conjecture \ref{kawamata conj} for complete intersections Calabi--Yau in Fano manifolds, we only need to prove that
nef line bundles on those Fano manifolds have nontrivial global sections. Motivated by this, we define the following:
\begin{definition}\label{def of perfect Fano}
A Fano manifold $Y$ is \emph{perfect} if any nef line bundle on it has a nontrivial global section.
\end{definition}
Kawamata's conjecture \ref{kawamata conj0} predicts that any Fano manifold should be perfect.
At the moment, we only show that there are lots of examples of perfect Fano manifolds.
\begin{proposition}\label{prop on perfect Fano} The following Fano manifolds are perfect:
\begin{enumerate}
\item toric Fano manifolds;
\item Fano manifolds of dimension $n\leq4$;
\item products of perfect Fano manifolds.
\end{enumerate}
\end{proposition}
\begin{proof}
(1) This is obvious from the fact that any nef line bundle on a complete toric variety is base-point-free (see \cite[Theorem 6.3.12]{cox}).

(2) We only deal with Fano 4-fold $X$ here. By Kawamata--Viehweg vanishing theorem, $H^{i}(X,L)=0$ for any nef line bundle $L$ on $X$ and $i>0$. Then Hirzebruch--Riemann--Roch formula gives
\begin{align*}h^{0}(X,L)=\chi(X,L)={}&\frac{1}{24}\int_{X} c^{4}_{1}(L)+\frac{1}{12}\int_{X} c^{3}_{1}(L) c_{1}(X)+\frac{1}{24}\int_{X}c^{2}_{1}(L) c^{2}_{1}(X)\\
  {}&+\frac{1}{24}\int_{X}c^{2}_{1}(L) c_{2}(X)+\frac{1}{24}\int_{X}c_{1}(L) c_{1}(X) c_{2}(X)+\chi(X,\mathcal{O}_{X}).
 \end{align*}
Fano condition implies that $\chi(X,\mathcal{O}_{X})=1$. By the pseudo-effectiveness of second Chern classes for Fano manifolds (see \cite[Theorem 6.1]{miyaoka}, \cite[Theorem 2.4]{kmmt}, and \cite[Theorem 1.3]{peternell}), those terms with $c_{2}(X)$ are non-negative. Hence
$h^{0}(X,L)>0$.

(3) For two Fano manifolds $X$ and $Y$, as $H^{1}(X,\mathcal{O}_{X})=0$, we know $\Pic(X\times Y)\cong \Pic(X)\times \Pic(Y)$ (see e.g. \cite[Ex. III.12.6]{hart}), i.e. a line bundle on $X\times Y$ is the box tensor of line bundles on $X$ and $Y$. Furthermore, it is obvious that the box tensor is nef if and only if so is each factor, and the box tensor has a global section if and only if so does each factor.
\end{proof}
To sum up, we verify Conjecture \ref{kawamata conj} for general complete intersection Calabi--Yau in perfect Fano manifolds.
\begin{theorem}\label{verification for CICY}
Let $Y$ be a perfect Fano manifold, $\{H_{i}\}_{i=1}^{m}$ a sequence of base-point-free ample line bundles on $Y$ such that $\otimes_{i=1}^{m}H_{i}\cong K_{Y}^{-1}$.
Let $X\subseteq Y$ be a general complete intersection of dimension $n\geq3$ defined by common zero locus of sections 
of $\{H_{i}\}_{i=1}^{m}$.

Then Conjecture \ref{kawamata conj} is true for $X$.
\end{theorem}
\begin{proof}
There is a sequence of projective manifolds
\begin{equation}X=X_{m}\subseteq X_{m-1}\subseteq\cdots\subseteq X_{1}\subseteq Y \nonumber \end{equation}
with $X_{i}=\cap_{j=1}^{i}s^{-1}_{j}(0)$, the common zero locus of sections of $\{H_{j}\}_{j=1}^{i}$. We inductively apply Theorem \ref{kollar thm} and Proposition \ref{compare sections} to $(X_{i},H_{i+1}|_{X_{i}})_{i=1}^{m}$, then any nef and big line bundle $L$ on $X$ comes from the restriction of a nef line bundle $L_Y$ on $Y$ and $h^{0}(X,L)>0$ since $h^0(Y, L_Y)>0$ by definition of perfect Fano manifolds.
\end{proof}

\subsection{$\Td_4$ of CICY}
In this subsection, we consider complete intersection Calabi--Yau varities (CICY, for short) in products of projective spaces.
Fix positive integers $n_1, \ldots, n_m$, a CICY $X$ in $\mathcal{P}(\bn)=\bP^{n_1}\times \cdots \times \bP^{n_m}$ is provided by the {\em configuration matrix}
$$
[\bn| \bq]= \left[
\begin{array}{c|ccc}
n_1 & q_1^1 & \cdots & q_K^1 \\
\vdots & \vdots & & \vdots \\
n_m & q_1^m & \cdots & q_K^m \\
\end{array}
\right]
$$
which encodes the dimensions of the ambient projective spaces and the (multi)-degrees of
the defining polynomials. Here $c_1(X)=0$ implies that
\begin{align}\label{CY condition}
\sum_{\alpha=1}^K q_\alpha^r=n_r+1
\end{align}
for all $1\leq r\leq m$. We say that $\bq>0$ if $q_\alpha^r>0$ for all $\alpha$ and $r$. Our goal is to prove the following:
\begin{theorem}\label{CICY td4}
Let $X=[\bn| \bq]\subset\mathcal{P}(\bn)$ be a general CICY with $\bq>0$. Then
$\td_4(X)$ is quasi-effective.
\end{theorem}

From now, $X=[\bn| \bq]\subset\mathcal{P}(\bn)$ is a general CICY with $\bq>0$. For each $r$, denote by $H_r$ the pull-back of hyperplane of $\bP^{n_r}$ to $\mathcal{P}(\bn)$, and $J_r=H_r|_X.$ It is easy to compute Chern classes of $X$ in terms of $J_r$ (cf. \cite[B.2.1]{M CY5}), for example, we have
\begin{align*}
c_2(X){}&=\sum_{r,s=1}^m c_2^{rs} J_r J_s \\
{}&= \sum_{r,s=1}^m \frac{1}{2}\left[-(n_r+1)\delta^{rs} + \sum_{\alpha=1}^K q_\alpha^r q_\alpha^s \right] J_r J_s,
 \end{align*}
 and
 \begin{align*}
     c_4 (X) {}&= \sum_{r,s,t,u=1}^m c_4^{rstu} J_r J_s J_t J_u \\
     {}&= \sum_{r,s,t,u=1}^m\frac{1}{4}\left[ -(n_r+1)\delta^{rstu} + \sum_{\alpha=1}^K q_\alpha^r q_\alpha^s q_\alpha^t q_\alpha^u + 2 c_2^{rs} c_2^{tu} \right] J_r J_s J_t J_u.
\end{align*}
Hence
\begin{align*}
2880\td_4(X)={}&4(3c_2^2(X)-c_4(X))\\
={}&\sum_{r,s,t,u=1}^m (12c_2^{rs}c_2^{tu} -4c_4^{rstu})J_r J_sJ_t J_u \\
={}& \sum_{r,s,t,u=1}^m\left[10c_2^{rs}c_2^{tu} +(n_r+1)\delta^{rstu} - \sum_{\alpha=1}^K q_\alpha^r q_\alpha^s q_\alpha^t q_\alpha^u \right] J_r J_s J_t J_u\\
={}& \sum_{r,s,t,u=1}^m\Bigg[\frac{5}{2}\left(-(n_r+1)\delta^{rs} + \sum_{\alpha=1}^K q_\alpha^r q_\alpha^s \right) \left(-(n_t+1)\delta^{tu} + \sum_{\alpha=1}^K q_\alpha^t q_\alpha^u \right)\\
{}&\qquad\qquad+(n_r+1)\delta^{rstu} - \sum_{\alpha=1}^K q_\alpha^r q_\alpha^s q_\alpha^t q_\alpha^u \Bigg] J_r J_s J_t J_u.
\end{align*}
Here we denote the coefficient of $J_r J_s J_t J_u $ in the above equation to be $A^{rstu}$. Rewrite the equation,
\begin{align}
\label{BBB}2880\td_4(X)={}&\sum_{r=1}^m B^{rrrr}J_r^4+ \sum_{\substack{1\leq r, s\leq m\\ r\neq s}} B^{rrrs}J_r^3J_s+ \sum_{1\leq r<s\leq m} B^{rrss}J_r^2J_s^2\\
{}&+\sum_{\substack{1\leq r,s,t\leq m\\ r\neq s, r\neq t, s<t}} B^{rrst}J_r^2J_sJ_t+\sum_{1\leq r<s<t<u\leq m} B^{rstu}J_r J_sJ_t J_u.\nonumber
\end{align}
Here by symmetry,
\begin{align*}
B^{rrrr}={}& A^{rrrr}=\frac{5}{2}\left(-(n_r+1) + \sum_{\alpha=1}^K (q_\alpha^r )^2\right)^2+(n_r+1)- \sum_{\alpha=1}^K (q_\alpha^r)^4;\\
B^{rrrs}={}&4A^{rrrs}=10\left(-(n_r+1) + \sum_{\alpha=1}^K (q_\alpha^r)^2 \right) \left( \sum_{\alpha=1}^K q_\alpha^r q_\alpha^s \right)- 4\sum_{\alpha=1}^K (q_\alpha^r)^3 q_\alpha^s;\\
B^{rrss}={}&2A^{rrss}+4A^{rsrs}\\
={}&10\left( \sum_{\alpha=1}^K q_\alpha^r q_\alpha^s \right)^2+5\left(-(n_r+1) + \sum_{\alpha=1}^K (q_\alpha^r)^2 \right)\left(-(n_s+1) + \sum_{\alpha=1}^K (q_\alpha^s)^2 \right)- 6\sum_{\alpha=1}^K (q_\alpha^r q_\alpha^s)^2;\\
B^{rrst}={}&4A^{rrst}+8A^{rsrt}\\
={}&10\left(-(n_r+1) + \sum_{\alpha=1}^K (q_\alpha^r)^2 \right) \left( \sum_{\alpha=1}^K q_\alpha^s q_\alpha^t \right)+20\left( \sum_{\alpha=1}^K q_\alpha^r q_\alpha^s \right)\left( \sum_{\alpha=1}^K q_\alpha^r q_\alpha^t \right)-12\sum_{\alpha=1}^K (q_\alpha^r)^2 q_\alpha^s q_\alpha^t;\\
B^{rstu}={}&8A^{rstu}+8A^{rtsu}+8A^{rust}.
\end{align*}
We have the following inequalities for these coefficients.
\begin{lemma}\label{B geq}In equation (\ref{BBB}),
\begin{enumerate}
\item $B^{rrrr}\geq 0$ unless $q_{\alpha_0}^r=2$ for some $\alpha_0$ and $q_\alpha^r=1$ for all $\alpha\neq \alpha_0$;
\item $B^{rrrs}\geq 0$ unless $q_\alpha^r=1$ for all $\alpha$;
\item $B^{rrss} \geq 0$;
\item $B^{rrst}\geq 0$;
\item $B^{rstu}\geq 0$.
\end{enumerate}
\end{lemma}
\begin{proof}Keep in mind that $q_\alpha^r$ are all positive integers and we will often use (\ref{CY condition}).

(1) Divide the index set into three part:
\begin{align*}
S_1={}&\{\alpha\mid q_{\alpha}^r=1\};\\
S_2={}&\{\alpha\mid q_{\alpha}^r=2\};\\
S_3={}&\{\alpha\mid q_{\alpha}^r\geq 3\}.
\end{align*}
The goal is to show that $B^{rrrr}\geq 0$ if and only if either $|S_3|>0$ or $|S_2|\ne 1$. The only if part is easy.
We show the if part. We have
\begin{align*}
{}&B^{rrrr}\\={}&\frac{5}{2}\left(-(n_r+1) + \sum_{\alpha=1}^K (q_\alpha^r )^2\right)^2+(n_r+1)- \sum_{\alpha=1}^K (q_\alpha^r)^4\\
={}&\frac{5}{2}\left( \sum_{\alpha\in S_2\cup S_3} \Big((q_\alpha^r )^2-q_\alpha^r\Big)\right)^2+\sum_{\alpha\in S_2\cup S_3} \Big(q_\alpha^r-(q_\alpha^r)^4\Big)\\
\geq {}&\frac{5}{2} \sum_{\alpha\in S_3} \Big((q_\alpha^r )^2-q_\alpha^r\Big)^2+
\frac{5}{2}\left( \sum_{\alpha\in S_2} \Big((q_\alpha^r )^2-q_\alpha^r\Big)\right)\left( \sum_{\alpha\in S_2\cup S_3} \Big((q_\alpha^r )^2-q_\alpha^r\Big)\right)+\sum_{\alpha\in S_2\cup S_3} \Big(q_\alpha^r-(q_\alpha^r)^4\Big)\\
\geq{}&\frac{5}{2}\left( \sum_{\alpha\in S_2} \Big((q_\alpha^r )^2-q_\alpha^r\Big)\right)\left( \sum_{\alpha\in S_2\cup S_3} \Big((q_\alpha^r )^2-q_\alpha^r\Big)\right)+\sum_{\alpha\in S_2} \Big(q_\alpha^r-(q_\alpha^r)^4\Big)\\
={}&5|S_2|\left( 2|S_2|+ \sum_{\alpha\in S_3} \Big((q_\alpha^r )^2-q_\alpha^r\Big)\right)-14|S_2|.
\end{align*}
Here the first inequality is just easy computation and for the second inequality, we use the fact that if $q\geq 3$, then
$$
\frac{5}{2}(q^2-q)^2+q-q^4=\frac{5}{2}q(q-1)(3q^2-7q-2)>0.
$$
If $|S_3|>0$, then $\sum_{\alpha\in S_3} ((q_\alpha^r )^2-q_\alpha^r)\geq 6$ and hence the above inequality gives
$$
B^{rrrr}\geq 30|S_2|-14|S_2|\geq 0.
$$
If $|S_2|\ne 1$, then the above inequality gives
$$
B^{rrrr}\geq 10|S_2|^2-14|S_2|\geq 0.
$$

(2) Divide the index set into two part:
\begin{align*}
S_1={}&\{\alpha\mid q_{\alpha}^r=1\};\\
S^\prime_2={}&\{\alpha\mid q_{\alpha}^r\geq 2\}.
\end{align*}
The goal is to show that $B^{rrrs}\geq 0$ if and only if $|S^\prime_2|>0$. The only if part is easy.
We show the if part. Assume that $|S^\prime_2|>0$, then
\begin{align*}
B^{rrrs}={}&10\left(-(n_r+1) + \sum_{\alpha=1}^K (q_\alpha^r)^2 \right) \left( \sum_{\alpha=1}^K q_\alpha^r q_\alpha^s \right)- 4\sum_{\alpha=1}^K (q_\alpha^r)^3 q_\alpha^s\\
={}&10\left(\sum_{\alpha=1}^K \Big((q_\alpha^r)^2-q_\alpha^r\Big) \right) \left( \sum_{\alpha=1}^K q_\alpha^r q_\alpha^s \right)- 4\sum_{\alpha=1}^K (q_\alpha^r)^3 q_\alpha^s\\
={}&10\left(\sum_{\alpha\in S^\prime_2} \Big((q_\alpha^r)^2-q_\alpha^r\Big) \right) \left( \sum_{\alpha=1}^K q_\alpha^r q_\alpha^s \right)- 4\sum_{\alpha=1}^K (q_\alpha^r)^3 q_\alpha^s\\
\geq{}& 5\left(\sum_{\alpha\in S^\prime_2} (q_\alpha^r)^2 \right) \left( \sum_{\alpha=1}^K q_\alpha^r q_\alpha^s \right)- 4\sum_{\alpha\in S_1\cup S^\prime_2} (q_\alpha^r)^3 q_\alpha^s\\
\geq{}&\left(\sum_{\alpha\in S^\prime_2} (q_\alpha^r)^2 \right) \left( \sum_{\alpha=1}^K q_\alpha^r q_\alpha^s \right)- 4\sum_{\alpha\in S_1} (q_\alpha^r)^3 q_\alpha^s\\
\geq{}&4 \left( \sum_{\alpha=1}^K q_\alpha^r q_\alpha^s \right)- 4\sum_{\alpha\in S_1} q_\alpha^s\geq 0.
\end{align*}
Here for the first inequality we use the fact that $2(q^2-q)\geq q^2$ for $q\geq 2$, the second is just easy computation, and for the last inequality we use the fact that $\sum_{\alpha\in S^\prime_2} (q_\alpha^r)^2\geq 4$ since $|S^\prime_2|>0$.

(3) Recall that $$-(n_r+1) + \sum_{\alpha=1}^K (q_\alpha^r )^2=\sum_{\alpha=1}^K \Big((q_\alpha^r )^2-q_\alpha^r\Big)\geq 0.$$
 It is easy to see that
\begin{align*}B^{rrss}\geq{}&10\left( \sum_{\alpha=1}^K q_\alpha^r q_\alpha^s \right)^2- 6\sum_{\alpha=1}^K (q_\alpha^r q_\alpha^s)^2\\
\geq {}&4\left( \sum_{\alpha=1}^K q_\alpha^r q_\alpha^s \right)^2\geq 0.
\end{align*}

(4) Similarly, it is easy to see that
\begin{align*}B^{rrst}\geq {}&20\left( \sum_{\alpha=1}^K q_\alpha^r q_\alpha^s \right)\left( \sum_{\alpha=1}^K q_\alpha^r q_\alpha^t \right)-12\sum_{\alpha=1}^K (q_\alpha^r)^2 q_\alpha^s q_\alpha^t\\
\geq{}& 8\left( \sum_{\alpha=1}^K q_\alpha^r q_\alpha^s \right)\left( \sum_{\alpha=1}^K q_\alpha^r q_\alpha^t \right)\geq 0.
\end{align*}

(5) For $r<s<t<u$, it is easy to see that
\begin{align*}A^{rstu} ={}&\frac{5}{2}\left(\sum_{\alpha=1}^K q_\alpha^r q_\alpha^s \right) \left(\sum_{\alpha=1}^K q_\alpha^t q_\alpha^u \right)- \sum_{\alpha=1}^K q_\alpha^r q_\alpha^s q_\alpha^t q_\alpha^u\\
\geq {}&\frac{5}{2}\sum_{\alpha=1}^K q_\alpha^r q_\alpha^s q_\alpha^t q_\alpha^u - \sum_{\alpha=1}^K q_\alpha^r q_\alpha^s q_\alpha^t q_\alpha^u\\
\geq {}&0.
\end{align*}
Similarly, $A^{rtsu}\geq 0$ and $A^{rust}\geq 0$. Hence $B^{rstu}\geq 0$.
\end{proof}

Note that, since $J_r$'s are nef divisors by definition, it is easy to see that if all the coefficients in (\ref{BBB}) are non-negative, then $\td_4(X)$ is quasi-effective. We need to deal with the case when the coefficients are not non-negative in the following two lemmas. Also note that since we consider $X$ to be a general CICY in $\PP(\bn)$, the nef cones of $X$ and $\PP(\bn)$ concides by Theorem \ref{kollar thm}
 (see proof of Proposition \ref{verification for CICY}). Hence to check the quasi-effectivity of a cycle on $X$, we can view it as a cycle on $\PP(\bn)$ since quasi-effectivity is tested by nef divisors. We will always use this observation.

For two cycles $C$ and $C'$, we will write $C\succeq C'$ if $C-C'$ is quasi-effective. We denote the index set $\RR:=\{r\mid q_\alpha^r=1 \text{ for all } \alpha\}$.

\begin{lemma}\label{Brrrr}If $r\not \in \RR$,
$$
B^{rrrr}J_r^4+ \sum_{\substack{1\leq s \leq m \\ s\neq r}} B^{rrrs}J_r^3J_s\succeq 0.
$$
\end{lemma}
\begin{proof}
By Lemma \ref{B geq}, $B^{rrrs}\geq 0$. Hence if $B^{rrrr}\geq 0$, there is nothing to prove. We may assume $B^{rrrr}< 0$. By Lemma \ref{B geq}, after reordering the index, we have $q_{1}^r=2$ and $q_\alpha^r=1$ for all $\alpha\geq 2$. In this case $B^{rrrr}=-4$ and
$$
B^{rrrs}=8q_1^s+16\sum_{\alpha=2}^K q_\alpha^s.
$$
Also we have $n_r+1=\sum_{\alpha=1}^Kq_\alpha^r=K+1$.
Viewing this cycle as a cycle on $\PP(\bn)$, we have
\begin{align*}
{}&B^{rrrr}J_r^4+ \sum_{\substack{1\leq s \leq m \\ s\neq r}} B^{rrrs}J_r^3J_s\\
={}&-4J_r^4+ \sum_{\substack{1\leq s \leq m \\ s\neq r}} \left(8q_1^s+16\sum_{\alpha=2}^K q_\alpha^s\right)J_r^3J_s\\
={}&\left(-4H_r^4+ \sum_{\substack{1\leq s \leq m \\ s\neq r}} \left(8q_1^s+16\sum_{\alpha=2}^K q_\alpha^s\right)H_r^3H_s\right)\cdot \prod_{\alpha=1}^K\left(\sum_{s=1}^m q_\alpha^sH_s\right)\\
={}&4\left(-H_r^4+ \sum_{\substack{1\leq s \leq m \\ s\neq r}} \left(2q_1^s+4\sum_{\alpha=2}^K q_\alpha^s\right)H_r^3H_s\right)\cdot \left(2H_r+\sum_{s\neq r} q_1^sH_s\right)\cdot \prod_{\alpha=2}^K\left(H_r+\sum_{s\ne r} q_\alpha^sH_s\right)\\
={}&8H_r^3\left(-H_r+ \sum_{\substack{1\leq s \leq m \\ s\neq r}} \left(2q_1^s+4\sum_{\alpha=2}^K q_\alpha^s\right)H_s\right)\cdot \left(H_r+\sum_{s\neq r} \frac{q_1^s}{2}H_s\right)\cdot \prod_{\alpha=2}^K\left(H_r+\sum_{s\ne r} q_\alpha^sH_s\right)\\
\succeq {}&8H_r^3\left(-H_r+ \sum_{\substack{1\leq s \leq m \\ s\neq r}} \left(\frac{1}{2}q_1^s+\sum_{\alpha=2}^K q_\alpha^s\right)H_s\right)\cdot \left(H_r+\sum_{s\neq r} \frac{q_1^s}{2}H_s\right)\cdot \prod_{\alpha=2}^K\left(H_r+\sum_{s\ne r} q_\alpha^sH_s\right).
\end{align*}
Note that by Lemma \ref{positive polynomial},
$$
\left(-H_r+ \sum_{\substack{1\leq s \leq m \\ s\neq r}} \left(\frac{1}{2}q_1^s+\sum_{\alpha=2}^K q_\alpha^s\right)H_s\right)\cdot \left(H_r+\sum_{s\neq r} \frac{q_1^s}{2}H_s\right)\cdot \prod_{\alpha=2}^K\left(H_r+\sum_{s\ne r} q_\alpha^sH_s\right)+H_r^{K+1}
$$
is a polynomial in terms of $H_1,\ldots, H_m$ with non-negative coefficients and note that $$H_r^{K+1}=0$$ since $K=n_r$ and $H_r$ is the pullback of hyperplane on $\bP^{n_r}$. Hence we have written $$B^{rrrr}J_r^4+ \sum_{\substack{1\leq s \leq m \\ s\neq r}} B^{rrrs}J_r^3J_s$$ as a polynomial in terms of $H_1,\ldots, H_m$ with non-negative coefficients, which is clearly quasi-effective.
\end{proof}
\begin{lemma}\label{Brrrs}
If $r\in \RR$,
$$
\sum_{\substack{1\leq s\leq m\\s\neq r}} B^{rrrs}J_r^3J_s+ \sum_{\substack{s\not \in \RR\\ s\neq r}} B^{rrss}J_r^2J_s^2+\sum_{\substack{s\in \RR\\ s\neq r}} \frac{1}{2}B^{rrss}J_r^2J_s^2+\sum_{\substack{1\leq s<t\leq m\\ s\neq r, t\neq r}} B^{rrst}J_r^2J_sJ_t\succeq 0.
$$
Here for convenience, we set $B^{rrss}=B^{ssrr}$ if $r>s$.
\end{lemma}
\begin{proof}In this case, $n_r+1=\sum_{\alpha=1}^Kq_\alpha^r=K$. Hence $J_r^K=(H_r|_X)^{n_r+1}=0$, since $H_r$ is the pullback of hyperplane on $\bP^{n_r}$. We may assume $K\geq 3$ otherwise there is nothing to prove. We have
\begin{align*}
B^{rrrs}={}&- 4\sum_{\alpha=1}^K q_\alpha^s;\\
B^{rrss}={}& 10\left( \sum_{\alpha=1}^K q_\alpha^s \right)^2- 6\sum_{\alpha=1}^K (q_\alpha^s)^2\geq 4\left( \sum_{\alpha=1}^K q_\alpha^s \right)^2;\\
B^{rrst}={}& 20\left( \sum_{\alpha=1}^Kq_\alpha^s \right)\left( \sum_{\alpha=1}^K q_\alpha^t \right)-12\sum_{\alpha=1}^K q_\alpha^s q_\alpha^t\geq 8\left( \sum_{\alpha=1}^Kq_\alpha^s \right)\left( \sum_{\alpha=1}^K q_\alpha^t \right).
\end{align*}
Moreover, since $K\geq 3$, if $s\in \RR$ and $s\ne r$, then
$$
B^{rrss}=10\left( \sum_{\alpha=1}^K q_\alpha^s \right)^2- 6\sum_{\alpha=1}^K (q_\alpha^s)^2=10K^2-6K\geq 8K^2\geq 8\left( \sum_{\alpha=1}^K q_\alpha^s \right)^2.
$$
Hence we have
\begin{align*}
{}&\sum_{\substack{1\leq s\leq m\\s\neq r}} B^{rrrs}J_r^3J_s+ \sum_{\substack{s\not \in \RR\\ s\neq r}} B^{rrss}J_r^2J_s^2+\sum_{\substack{s\in \RR\\ s\neq r}} \frac{1}{2}B^{rrss}J_r^2J_s^2+\sum_{\substack{1\leq s<t\leq m\\ s\neq r, t\neq r}} B^{rrst}J_r^2J_sJ_t\\
\succeq {}&-\sum_{\substack{1\leq s\leq m\\s\neq r}} 4\sum_{\alpha=1}^K q_\alpha^s J_r^3J_s+ \sum_{\substack{ s\neq r}} 4\left( \sum_{\alpha=1}^K q_\alpha^s \right)^2J_r^2J_s^2+\sum_{\substack{1\leq s<t\leq m\\ s\neq r, t\neq r}} 8\left( \sum_{\alpha=1}^Kq_\alpha^s \right)\left( \sum_{\alpha=1}^K q_\alpha^t \right)J_r^2J_sJ_t\\
={}&-4\sum_{\substack{1\leq s\leq m\\s\neq r}} \sum_{\alpha=1}^K q_\alpha^s J_r^3J_s+ 4\left(\sum_{\substack{1\leq s\leq m\\s\neq r}} \sum_{\alpha=1}^K q_\alpha^sJ_rJ_s\right)^2\\
={}&4J_r\left(\sum_{\substack{1\leq s\leq m\\s\neq r}} \sum_{\alpha=1}^K q_\alpha^sJ_rJ_s\right)\left(-J_r+\sum_{\substack{1\leq s\leq m\\s\neq r}} \sum_{\alpha=1}^K q_\alpha^sJ_s\right).
\end{align*}
Viewing as a cycle on $\PP(\bn)$ and use the fact that $H_r^{K+1}=0$,
\begin{align*}
{}& -J_r+\sum_{\substack{1\leq s\leq m\\s\neq r}} \sum_{\alpha=1}^K q_\alpha^sJ_s\\
={}&\left(-H_r+\sum_{\substack{1\leq s\leq m\\s\neq r}} \sum_{\alpha=1}^K q_\alpha^sH_s\right)\cdot \prod_{\alpha=1}^K\left(\sum_{s=1}^m q_\alpha^sH_s\right)\\
={}&\left(-H_r+\sum_{\substack{1\leq s\leq m\\s\neq r}} \sum_{\alpha=1}^K q_\alpha^sH_s\right)\cdot \prod_{\alpha=1}^K\left(H_r+\sum_{s\ne r} q_\alpha^sH_s\right)\\
={}&\left(-H_r+\sum_{\substack{1\leq s\leq m\\s\neq r}} \sum_{\alpha=1}^K q_\alpha^sH_s\right)\cdot \prod_{\alpha=1}^K\left(H_r+\sum_{s\ne r} q_\alpha^sH_s\right)+H_r^{K+1}
\end{align*}
can be expressed as a polynomial in terms of $H_1,\ldots, H_m$ with non-negative coefficients by Lemma \ref{positive polynomial}. Hence we may
express
$$
\sum_{\substack{1\leq s\leq m\\s\neq r}} B^{rrrs}J_r^3J_s+ \sum_{\substack{s\not \in \RR\\ s\neq r}} B^{rrss}J_r^2J_s^2+\sum_{\substack{s\in \RR\\ s\neq r}} \frac{1}{2}B^{rrss}J_r^2J_s^2+\sum_{\substack{1\leq s<t\leq m\\ s\neq r, t\neq r}} B^{rrst}J_r^2J_sJ_t
$$
as a polynomial in terms of $H_1,\ldots, H_m$ with non-negative coefficients, which is clearly quasi-effective.
\end{proof}

\begin{proof}[Proof of Theorem \ref{CICY td4}]By (\ref{BBB}), Lemmas \ref{B geq}, \ref{Brrrr}, and \ref{Brrrs},
\begin{align*}
{}&2880\td_4(X)\\={}&\sum_{r=1}^m B^{rrrr}J_r^4+ \sum_{\substack{1\leq r, s\leq m\\ r\neq s}} B^{rrrs}J_r^3J_s+ \sum_{1\leq r<s\leq m} B^{rrss}J_r^2J_s^2\\
{}&+\sum_{\substack{1\leq r,s,t\leq m\\ r\neq s, r\neq t, s<t}} B^{rrst}J_r^2J_sJ_t+\sum_{1\leq r<s<t<u\leq m} B^{rstu}J_r J_sJ_t J_u\\
\succeq {}&\sum_{r\not \in \RR} B^{rrrr}J_r^4+ \sum_{\substack{1\leq r, s\leq m\\ r\neq s}} B^{rrrs}J_r^3J_s+ \sum_{1\leq r<s\leq m} B^{rrss}J_r^2J_s^2+\sum_{\substack{r \in \RR\\ 1\leq s,t\leq m\\ r\neq s, r\neq t, s<t}} B^{rrst}J_r^2J_sJ_t\\
= {}&\sum_{r\not \in \RR} B^{rrrr}J_r^4+ \sum_{r\not \in \RR}\sum_{\substack{1\leq s \leq m \\ s\neq r}} B^{rrrs}J_r^3J_s\\{}&+\sum_{r \in \RR}\sum_{\substack{1\leq s \leq m \\ s\neq r}} B^{rrrs}J_r^3J_s+ \sum_{1\leq r<s \leq m} B^{rrss}J_r^2J_s^2+\sum_{\substack{r \in \RR\\ 1\leq s,t\leq m\\ r\neq s, r\neq t, s<t}} B^{rrst}J_r^2J_sJ_t\\
\succeq {}&\sum_{r\not \in \RR} \left[B^{rrrr}J_r^4+ \sum_{ s\neq r} B^{rrrs}J_r^3J_s\right]\\{}&+\sum_{r \in \RR}\left[\sum_{s\neq r} B^{rrrs}J_r^3J_s+ \sum_{\substack{s\not \in \RR\\ s\neq r}} B^{rrss}J_r^2J_s^2+\sum_{\substack{s\in \RR\\ s\neq r}} \frac{1}{2}B^{rrss}J_r^2J_s^2+\sum_{\substack{1\leq s<t\leq m\\ s\neq r, t\neq r}} B^{rrst}J_r^2J_sJ_t\right]\\
\succeq{}& 0.
\end{align*}
Here for the first and second inequalities we use Lemma \ref{B geq}, and for the last inequality we use Lemmas \ref{Brrrr} and \ref{Brrrs}.
\end{proof}
\section{Some effective results}
In this section, using Hirzebruch--Riemann--Roch formula and Miyaoka--Yau inequality \cite{miyaoka, yau0}, we prove a weaker version of Conjecture \ref{kawamata conj} in all dimensions (which is related to a conjecture of Beltrametti and Sommese, see for instance H\"{o}ring's work \cite{horing}).

We first consider odd dimensions.
\begin{theorem}\label{HRR for odd CY}
Let $X$ be a smooth projective variety of dimension $2k+1$ ($k\geq1$) with $c_1(X)=0$ in $H^2(X, \bR)$ and $L$ a nef and big line bundle on $X$. Then there exists
$i\in\{1,2,\ldots, k\}$ such that $H^{0}(X,L^{\otimes i})\neq0$.
\end{theorem}
\begin{proof}
By contradiction, we assume $h^{0}(X,L^{\otimes i})=0$ for $i=1,2,\ldots,k$,
then $\chi(X,L^{\otimes i})=0$ by Kawamata--Viehweg vanishing theorem.
Hirzebruch--Riemann--Roch formula gives
\begin{equation}f(t)\triangleq \chi(X,L^{\otimes t})=\int_{X}\frac{L^{2k+1}}{(2k+1)!}t^{2k+1}+\int_{X}\frac{L^{2k-1}\cdot \td_{2}(X)}{(2k-1)!}t^{2k-1}
+\cdots+\int_{X}(L\cdot \td_{2k}(X))t. \nonumber \end{equation}
as $\td_{\text{odd}}(X)=0$. Then $f(-t)=-f(t)$ and degree $(2k+1)$-polynomial $f(t)$ has roots $\{0,\pm1,\pm2,\ldots,\pm k\}$.
Then we can write
\begin{equation}f(t)=\alpha t(t^{2}-1)(t^{2}-2^{2})\cdots(t^{2}-k^{2}), \nonumber \end{equation}
where $\alpha=\int_{X}\frac{L^{2k+1}}{(2k+1)!}>0$. The coefficient of $t^{2k-1}$ is $-\alpha\cdot(\sum_{i=1}^{k}i^{2})=\int_{X}\frac{L^{2k-1}\cdot c_{2}(X)}{12(2k-1)!}$, where we get a contradiction as the RHS is non-negative by the Miyaoka--Yau inequality \cite{miyaoka, yau0}.
\end{proof}
Then we consider even dimensions.
\begin{theorem}\label{HRR for even CY}
Let $X$ be a smooth projective variety of dimension $4k+2$ or $4k+4$ ($k\geq0$) with $c_1(X)=0$ in $H^2(X, \bR)$ and $L$ a nef and big line bundle on $X$. Then there exists
$i\in\{1,2,\ldots,2k+1\}$ such that $H^{0}(X,L^{i})\neq0$.
\end{theorem}
\begin{proof}
If $\dim X=4k+2$, assume $h^{0}(X,L^{\otimes i})=0$ for $i=1,2,\ldots,2k+1$, then $\chi(X,L^{\otimes i})=0$.
Hirzebruch--Riemann--Roch formula gives
\begin{equation} f(t)\triangleq \chi(X,L^{\otimes t})=\int_{X}\frac{L^{4k+2}}{(4k+2)!}t^{4k+2}+\int_{X}\frac{L^{4k}\cdot \td_{2}(X)}{(4k)!}t^{4k}
+\cdots+\chi(X,\mathcal{O}_{X}). \nonumber \end{equation}
as $\td_{\text{odd}}(X)=0$. Then $f(-t)=f(t)$ and degree $(4k+2)$-polynomial $f(t)$ has roots $\{\pm1,\pm2,\ldots,\pm (2k+1)\}$.
Then we can write
\begin{equation}f(t)=\alpha (t^{2}-1)(t^{2}-2^{2})\cdots(t^{2}-(2k+1)^{2}), \nonumber \end{equation}
where $\alpha=\int_{X}\frac{L^{4k+2}}{(4k+2)!}>0$. The coefficient of $t^{4k}$ is $-\alpha\cdot(\sum_{i=1}^{2k+1}i^{2})=\int_{X}\frac{L^{4k}\cdot c_{2}(X)}{12(4k)!}$, where we get a contradiction as the RHS is non-negative by the Miyaoka--Yau inequality \cite{miyaoka, yau0}.

If $\dim X=4k+4$, we similarly assume $f(t)=\chi(X,L^{\otimes t})$ has roots $\{\pm1,\pm2,\ldots,\pm (2k+1)\}$, and then
\begin{equation}f(t)=\alpha (t^{2}-1)(t^{2}-2^{2})\cdots(t^{2}-(2k+1)^{2})(t^{2}-\beta) \nonumber \end{equation}
for some $\beta\in\mathbb{C}$ and $\alpha=\int_{X}\frac{L^{4k+4}}{(4k+4)!}>0$. The coefficient of $t^{4k+2}$ is
$-\alpha\cdot(\beta+\sum_{i=1}^{2k+1}i^{2})=\int_{X}\frac{L^{4k+2}\cdot c_{2}(X)}{12(4k+2)!}$, and the constant term is
$\alpha\cdot\beta\cdot((2k+1)!)^{2}=\chi(X,\mathcal{O}_{X})$. Miyaoka--Yau inequality gives $\beta<0$,
which contradicts to $\chi(X,\mathcal{O}_{X})\geq0$ by Theorem \ref{classification of CY}.
\end{proof}
\vspace{1pt}
\begin{remark}  ~\\
1. As a corollary, Conjecture \ref{kawamata conj} holds true in dimension $n\leq4$. \\
2. For a hyperk\"{a}hler variety $X$ of dimension $2n$ ($n\geq 2$) and $L$ a nef and big line bundle on $X$, we can enhance the above result by using the effectiveness of fourth Todd class (Theorem \ref{effective of Td4}), and show that there exists a positive integer $i\leq \rounddown{\frac{n-2}{2}}+\rounddown{\frac{n}{2}}$ such that $H^{0}(X,L^{\otimes i})\ne 0$. We leave the detail to the readers.
\end{remark}

\section{Appendix}
In the appendix, we prove some basic lemmas.
\begin{lemma}\label{lemma on integers}
Let $n\in\mathbb{Z}_{>0}$ be a positive integer and $\lambda\in\mathbb{Q}$ be a rational number. \\
(1) If $n\geq2$ and $(n+1)\cdot\binom{\lambda+n}{n }$ is an integer, then $\lambda\in\mathbb{Z}$. \\
(2) If $\binom{\lambda+n+1 }{n }$ is an integer, then $\lambda\in\mathbb{Z}$.
\end{lemma}
\begin{proof}
(1) Write $\lambda=p/q$ for $p\in \mathbb{Z}$ and $q\in\mathbb{Z}_{>0}$ with $\text{gcd}(p,q)=1$. By 
contradiction, we may assume $q>1$.
Then $(n+1)\cdot\binom{\lambda+n}{n }\in \ZZ$ implies that
\begin{align}
n!q^n\mid (n+1)(p+nq)\cdots(p+q).\label{n1}
\end{align}

We claim that either $(n+1)$ is prime or $(n+1)\mid2\cdot n!$. If $n\leq 4$, it is obvious. Assume that $n\geq5$.
If $(n+1)$ is not prime, we have a factorization $(n+1)=a\cdot b$, for integers $2\leq a,b\leq n$. If $a\neq b$, they both appear in $n!$ as factors, and hence $(n+1)\mid n!$.
If $a=b$, then $a=\sqrt{n+1}\leq \frac{n}{2}$, so $2a\leq n$. Hence $a$ and $2a$ appear in $n!$ as factors, and hence $a^2\mid n!$.

For the case when $(n+1)$ is prime, (\ref{n1}) implies that
$$q^{n}\mid (n+1)(p+nq)\cdots(p+q), $$
which implies that $q\mid (n+1)p^{n}$ and $q^{2}\mid (n+1)(p^{n}+\frac{n(n+1)}{2}p^{n-1}q)$. As $(p,q)=1$ and
$(n+1)$ is prime, we conclude from the first dividing relation that $q=n+1$. Combining with the second relation, we get $q^{2}\mid q p^{n}$,
which contradicts with $(p,q)=1$.

For the case when $(n+1)\mid 2\cdot n!$, (\ref{n1}) implies that
$$q^{n}\mid 2(p+nq)\cdots(p+q).$$
which implies that $q\mid 2p^{n}$. As $(p,q)=1$, we conclude that $q=2$ and $p$ is odd. Hence $(p+nq)\cdots(p+q)$ is odd and (\ref{n1}) implies that $2^{n}\mid n+1$, which is absurd for $n\geq 2$.

(2) Write $\lambda=p/q$ for $p\in \mathbb{Z}$ and $q\in\mathbb{Z}_{>0}$ with $\text{gcd}(p,q)=1$ Then $\binom{\lambda+n+1 }{n }\in \ZZ$ implies that \begin{equation}q^{n}\mid (p+(n+1)q)(p+nq)\cdots(p+2q), \nonumber \end{equation}
and hence $q\mid p^{n}$. Since $(p,q)=1$, this implies that $q=1$ and $\lambda$ is an integer.
\end{proof}

\begin{lemma}\label{positive polynomial}Let $m,K$ be two positive integers and $\{q_\alpha^s\mid 1\leq \alpha\leq K, 1\leq s\leq m\}$ a set of non-negative numbers. Consider the homogenous polynomial
$$
f(x_1,\ldots, x_m)=\left(-x_1+\sum_{\alpha=1}^K\sum_{s=2}^mq _\alpha^s x_s\right)\cdot \prod_{\alpha=1}^K \left(x_1+\sum_{s=2}^mq_\alpha^s x_s\right)+x_1^{K+1}.
$$
Then all coefficients of $f$ are non-negative.
\end{lemma}
\begin{proof}
Consider $f$ as a polynomial of $x_1$ with coefficients in terms of $x_2,\ldots, x_m$. We need to show that for $1\leq k\leq K+1$, the coefficient of $x_1^k$ is a polynomial in terms of $x_2,\ldots, x_m$ with non-negative coefficients. It is easy to see that the coefficient of $x_1^{K+1}$ and $x_1^{K}$ are $0$. Fix $1\leq k\leq K$, then the coefficient of $x_1^{K-k}$ is
\begin{align*}
{}& \left(\sum_{\alpha=1}^K\sum_{s=2}^mq _\alpha^s x_s\right)\cdot \left(\sum_{\alpha_1<\cdots <\alpha_k}\prod_{j=1}^k\sum_{s=2}^mq_{\alpha_j}^s x_s\right)- \sum_{\alpha_1<\cdots <\alpha_{k+1}}\prod_{j=1}^{k+1}\sum_{s=2}^mq_{\alpha_j}^s x_s\\
={}&\sum_{\alpha_1<\cdots <\alpha_k}\left( \left(\sum_{\alpha=1}^K\sum_{s=2}^mq _\alpha^s x_s\right)\cdot \left(\prod_{j=1}^k\sum_{s=2}^mq_{\alpha_j}^s x_s\right) - \sum_{\alpha_{k+1}>\alpha_k}\prod_{j=1}^{k+1}\sum_{s=2}^mq_{\alpha_j}^s x_s\right)\\
={}&\sum_{\alpha_1<\cdots <\alpha_k}\left( \left(\sum_{\alpha=1}^K\sum_{s=2}^mq _\alpha^s x_s\right)\cdot \left(\prod_{j=1}^k\sum_{s=2}^mq_{\alpha_j}^s x_s\right) - \left(\sum_{\alpha_{k+1}>\alpha_k}\sum_{s=2}^mq_{\alpha_{k+1}}^s x_s\right)\cdot\left(\prod_{j=1}^{k}\sum_{s=2}^mq_{\alpha_j}^s x_s\right)\right)\\
={}&\sum_{\alpha_1<\cdots <\alpha_k}\left( \left(\sum_{\alpha\leq \alpha_k}\sum_{s=2}^mq _\alpha^s x_s\right)\cdot \left(\prod_{j=1}^k\sum_{s=2}^mq_{\alpha_j}^s x_s\right)\right)
\end{align*}
which is a polynomial in terms of $x_2,\ldots, x_m$ with non-negative coefficients.
\end{proof}

\end{document}